\documentclass[10pt,a4paper,twoside]{amsart}
\usepackage{amsmath,amsfonts,amsthm,amsopn,color,amssymb,enumitem}
\usepackage{palatino}
\usepackage{graphicx}
\usepackage[colorlinks=true]{hyperref}
\hypersetup{urlcolor=blue, citecolor=red, linkcolor=blue}

\usepackage{cite}

\usepackage[colorinlistoftodos]{todonotes}

\makeatletter
\providecommand\@dotsep{5}
\def\listtodoname{List of Todos}
\def\listoftodos{\@starttoc{tdo}\listtodoname}
\makeatother

\newcommand{\e}{\varepsilon}

\newcommand{\C}{\mathbb{C}}
\newcommand{\R}{\mathbb{R}}

\newcommand{\RD}{{\mathbb{R}^2}}

\newcommand{\weakto}{\rightharpoonup}

\renewcommand{\le}{\leslant}
\renewcommand{\ge}{\geslant}
\renewcommand{\a }{\alpha }

\renewcommand{\b }{\beta }

\renewcommand{\d }{\delta }

\newcommand{\g }{\gamma }

\renewcommand{\l }{\lambda}
\newcommand{\n }{\nabla }
\newcommand{\s }{\sigma }

\renewcommand{\H}{H^1(\RD)}
\newcommand{\Hr}{H^1_r(\RD)}

\newcommand{\CH}{\mathcal{H}}
\newcommand{\I}{\mathcal{I}}

\newcommand{\N}{\mathbb{N}}

\renewcommand{\C}{\mathbb{C}}
\renewcommand{\o}{\omega}

\def\bbm[#1]{\mbox{\boldmath $#1$}}
\newcommand{\beq }{\begin{equation}}
\newcommand{\eeq }{\end{equation}}

\renewcommand{\le}{\leqslant}
\renewcommand{\ge}{\geqslant}
\newcommand{\dis}{\displaystyle}

\newcommand{\ir}{\int_{-\infty}^{+\infty}}
\newcommand{\ird}{\int_{\R^2}}

\newtheorem{theorem}{Theorem}[section]
\newtheorem{lemma}[theorem]{Lemma}

\newtheorem{proposition}[theorem]{Proposition}

\title[A Gauged Nonlinear Schr\"{o}dinger Equation including a vortex point]{Standing waves for a
 Gauged Nonlinear Schr\"{o}dinger Equation with a vortex point}

\author[Jiang]{Yongsheng Jiang$^1$}
\address{$^1$School of Statistics and Mathematics, Zhongnan University of Economics and Law,  430073 Wuhan, P.R.China}

\author[Pomponio]{Alessio Pomponio$^2$}
\address{$^2$Dipartimento di Meccanica, Matematica e Management, Politecnico di
Bari, Via E. Orabona 4, 70125 Bari, Italy.}
\author[Ruiz]{David Ruiz$^3$}
\address{$^3$Dpto. An\'{a}lisis Matem\'{a}tico, Avda. Fuentenueva s/n, 18071 Granada (Spain).}

\thanks{Y. Jiang is supported by NSFC (11201486).  A.P. is supported by M.I.U.R. - P.R.I.N.
``Metodi variazionali e topologici nello studio di fenomeni non
lineari'', by GNAMPA Project ``Aspetti differenziali e geometrici nello studio
di problemi ellittici quasilineari". D.R. is supported by the Grant MTM2011-26717
and by FQM-116 (J. Andaluc{\'\i}a).}

\email{jiangys@znufe.edu.cn, alessio.pomponio@poliba.it, daruiz@ugr.es}
\date{}

\keywords{Gauged Schr\"{o}dinger Equations, Chern-Simons theory,
Variational methods, concentration compactness.}

\subjclass[2010]{35J20, 35Q55.}

\begin{document}

\begin{abstract}

This paper is motivated by a gauged Schr\"{o}dinger equation in
dimension 2. We are concerned with radial stationary states under
the presence of a vortex at the origin. Those states solve a
nonlinear nonlocal PDE with a variational structure. We will
study the global behavior of that functional, extending known
results for the regular case.
\end{abstract}

\maketitle

\section{Introduction}

In this paper we are concerned with a planar gauged Nonlinear
Schr\"{o}dinger Equation:
\beq \label{planar}
i D_0\phi +
(D_1D_1 +D_2D_2)\phi + |\phi|^{p-1} \phi =0.
\eeq

Here $t \in \R$, $x=(x_1, x_2) \in \R^2$, $\phi :
\R \times \R^2 \to \C$ is the scalar field, $A_\mu : \R\times \R^2
\to \R$ are the components of the gauge potential and $D_\mu =
\partial_\mu + i A_\mu$ is the covariant derivative ($\mu = 0,\
1,\ 2$).

The modified gauge field equation proposes the following equation
for the gauge potential, including the so-called Chern-Simons term
(see \cite{dunne, tar}):

\beq \label{hello} 
\partial_\mu F^{\mu \nu} + \frac 1 2
\kappa \epsilon^{\nu \alpha \beta} F_{\alpha \beta}= j^\nu, \ \mbox{ with } \ F_{\mu \nu}=
\partial_\mu A_\nu - \partial_\nu A_\mu.
\eeq

In the above equation, $\kappa$ is a parameter that
measures the strength of the Chern-Simons term. As usual,
$\epsilon^{\nu \alpha \beta}$ is the Levi-Civita tensor, and
super-indices are related to the Minkowski metric with signature
$(1,-1,-1)$. Finally, $j^\mu$ is the conserved matter current,

$$ j^0= |\phi|^2, \ j^i= 2 {\rm Im} \left (\bar{\phi} D_i \phi \right).$$

At low energies, the Maxwell term in \eqref{hello} becomes
negligible and can be dropped, giving rise to: \beq \label{cs}
\frac 1 2 \kappa \epsilon^{\nu \alpha \beta} F_{\alpha \beta}=
j^\nu. \eeq See \cite{hagen, hagen2, jackiw0, jackiw, jackiw2} for
the discussion above. If we fix $\kappa = 2$, equations
\eqref{planar} and \eqref{cs} lead us to the problem:
 \beq \label{eq:e0}
\begin{array}{l}
i D_0\phi + (D_1D_1 +D_2D_2)\phi + |\phi|^{p-1} \phi =0,\\
\partial_0 A_1  -  \partial_1 A_0  = {\rm Im}( \bar{\phi}D_2\phi), \\
\partial_0 A_2  - \partial_2 A_0 = -{\rm Im}( \bar{\phi}D_1\phi), \\
\partial_1 A_2  -  \partial_2 A_1 =  \frac 1 2 |\phi|^2. \end{array}
\eeq

As usual in Chern-Simons theory, problem \eqref{eq:e0} is
invariant under gauge transformation,
\beq \label{gauge} \phi \to \phi e^{i\chi}, \quad A_\mu \to A_\mu
- \partial_{\mu} \chi, \eeq for any arbitrary $C^\infty$ function
$\chi$.

This model was first proposed and studied in \cite{jackiw0,
jackiw, jackiw2}, and sometimes has received the name of
Chern-Simons-Schr\"{o}dinger equation. The initial value problem,
well-posedness, global existence and blow-up, scattering, etc.
have been addressed in \cite{berge, huh, huh3, tataru, oh} for the
case $p=3$. See also \cite{liu} for a global existence result in
the defocusing case, and \cite{chen-smith} for a uniqueness result to the infinite radial hierarchy.

\

The existence of stationary states for \eqref{eq:e0} and general
$p>1$  has been studied in \cite{byeon} for the regular
case (see also \cite{cunha, huh2, AD, AD2}). Very recently, in \cite{byeon-pre} the case with a vortex point has been considered (with respect
to that paper, our notation interchanges the indices $1$ and
$2$). Consider the ansatz:
\begin{equation*}
\begin{array} {lll}
\phi = u(r) e^{i (N\theta+\omega t)}, && A_0= A_0(r),
\\
A_1=\dis -\frac{x_2}{r^2}h(r), && A_2= \dis \frac{x_1}{r^2}h(r)
\end{array}
\end{equation*}

Here ($r, \ \theta)$ are the polar coordinates of $\R^2$, and $N
\in \N \cup \{0\}$ is the order of the vortex at the origin ($N=0$
corresponds to the regular case).

In \cite{byeon-pre} it is found that $u$ solves the
equation:

 \[ 
 - \Delta u(x) + \o \, u+  \dis
\frac{(h_u(|x|)-N)^2}{|x|^2} u+ A_0(|x|)  u(x)  =|u(x)|^{p-1}u(x),
\quad x \in \RD,
\]
where  \beq \label{hu} h_u(r)= \frac 12\int_0^r s u^2(s) \,
ds,
\eeq and
$$A_0(r)= \xi+ \int_r^{+\infty} \frac{h_u(s)-N}{s} u^2(s)\, ds, \ \ \xi \in \R.$$
The value $\xi$ above appears as an integration constant. Without loss of generality, we can assume $\xi=0$; otherwise it suffices to use the gauge invariance \eqref{gauge} with $\chi = \xi t$.
Then, our problem becomes:

 \beq \label{equation}   - \Delta u(x) + \o \, u+  \dis
\frac{(h_u(|x|)-N)^2}{|x|^2} u+  \left ( \int_{|x|}^{+\infty}
\frac{h_u(s)-N}{s} u^2(s)\, ds \right )  u(x)  =|u(x)|^{p-1}u(x).\eeq

Observe that \eqref{equation} is a nonlocal equation. In
\cite{byeon-pre} it is shown that \eqref{equation} is
indeed the Euler-Lagrange equation of the energy functional $
I_{\o}: \CH \to \R$,
\begin{align*} \label{functional}
I_{\o}(u) & = \dis \frac 12 \int_{\R^2} \left(|\nabla u(x)|^2 + \o
u^2(x) \right) \, dx \nonumber
\\
 & \quad+\dis \frac{1}{2} \int_{\R^2}
\frac{u^2(x)}{|x|^2}\left(h_u(r) -N\right)^2 dx  - \dis
\frac{1}{p+1}  \int_{\R^2} |u(x)|^{p+1} \, dx.
\end{align*}
The Hilbert space $\CH$ is defined as: \beq \label{space}
\CH=\{u\in H^1_r(\R^2):\int_{\R}\frac{u^2(x)}{|x|^2}dx<+\infty\},
\eeq endowed by the norm
\[ 
\|u\|_\CH^2=\int_{\R^2}|\nabla u(x)|^2+\left
(1+\frac{1}{|x|^2}\right )u^2(x)\ dx.
\]

Let us observe that the energy functional $I_{\o}$ presents a
competition between the nonlocal term and the local nonlinearity
of power-type. The study of the behavior of the functional under
this competition is one of the main motivations of this paper. For
$p>3$, it is known that $I_\o$ is unbounded from below, so it
exhibits a mountain-pass geometry (see \cite{byeon, huh2} for the case
$N=0$ and \cite[Section 5]{byeon-pre} for $N \in \N$). In a
certain sense, in this case the local nonlinearity dominates the
nonlocal term. However the existence of a solution is not so
direct, since for $p \in (3,5)$ the (PS) property is not known to
hold. This problem is bypassed by combining the so-called
monotonicity trick of Struwe (\cite{struwe}) with a Pohozaev
identity.

A special case in the above equation is $p=3$: in this case,
solutions have been explicitly found in \cite{byeon, byeon-pre} as
optimizers of a certain inequality. An alternative approach would
be to pass to a self-dual equation, which leads to a Liouville
equation in $\R^2$, singular if $N>0$.

The situation is different if $p\in (1,3)$; here the nonlocal term
prevails over the local nonlinearity, in a certain sense. In
\cite{AD}, the second and third authors studied whether $I_\o$ is
bounded from below or not for $p\in (1,3)$ and $N=0$. The
situation happened to be quite rich and unexpected, and very
different from the usual nonlinear Schr\"{o}dinger equation. Indeed, the boundedness of $I_\o$ for $N=0$ depends on the phase
$\omega$ and the threshold value $\o_0$ is explicit, namely:

\beq \label{w0} \omega_0= \frac{3-p}{3+p}\ 3^{\frac{p-1}{2(3-p)}}\
2^{\frac{2}{3-p}} \left(
\frac{m^2(3+p)}{p-1}\right)^{-\frac{p-1}{2(3-p)}}, \eeq with \beq
\label{defm} m = \ir  \left ( \frac{2}{p+1} \cosh^2 \left
(\frac{p-1}{2} r \right) \right)^{\frac{2}{1-p}} \, dr.\eeq

The purpose of this paper is to extend such result to the case
$N>0$, which is more relevant from the point of view of the
applications. This study has been prompted by Remark 5.1 in
\cite{byeon-pre}.

Our main results are the following:

\begin{theorem} \label{teo1} For $\o_0$ as given in  \eqref{w0}, there holds:

\begin{enumerate}[label=(\roman*), ref=\roman*]
\item \label{en:<} if $\o\in (0,\o_0)$, then $I_\o$ is unbounded
from below;

\item \label{en:=} if $\o=\o_0$, then $I_{\o_0}$ is bounded from
below, not coercive and $\inf I_{\o_0}<0$;

\item  \label{en:>} if $\o>\o_0$, then $I_\o$ is bounded from
below and coercive.
\end{enumerate}
\end{theorem}

Regarding the existence of solutions, we obtain the following
results:

\begin{theorem} \label{teo2} There exist $\bar{\omega} > \tilde{\o}>\o_0$ such that:

\begin{enumerate}[label=(\roman*), ref=\roman*]
\item if $\o> \bar{\o}$, then \eqref{equation} has no solutions
different from zero;
\item if $\o\in (\o_0,\tilde{\o})$, then
\eqref{equation} admits at least two positive solutions: one of
them is a global minimizer for $I_\o$ and the other is a
mountain-pass solution;
\item for almost every $\o\in (0,\o_0)$
\eqref{equation} admits a positive solution.
\end{enumerate}
\end{theorem}

The proofs follow the same ideas as in \cite{AD}, and is related
to a natural limit problem. Roughly speaking, this limit problem
stems from the behavior of the map $ \rho \mapsto I_{\o}(u(\cdot-
\rho))$ as $\rho \to +\infty$, and this does not depend on $N$.
However, in our proofs the analysis made in Proposition
\ref{prop-fund} must be re-elaborated with respect to that of
\cite{AD}, and the new terms need new estimates in the asymptotic
expansions that follow afterwards. Moreover, the non-existence
result of Theorem \ref{teo2} is immediate for $N=0$ but its proof
becomes delicate for $N>0$. Finally, the case $N>0$ is more
relevant from the point of view of the Physics model, since it
includes a vortex at the origin. One of the main features of the
Chern-Simons theory is the appearance of vortices in the model,
see \cite{dunne, tar, yang}).

The rest of the paper is organized as follows. Section 2 is
devoted to some notations and preliminary results. In Section 3 we
prove Theorems \ref{teo1} and \ref{teo2}.

\section{Preliminaries}

Let us first fix some notations. We denote by $H_r^1(\R^2)$ the
Sobolev space of radially symmetric functions, and $\| \cdot \|$
its usual norm. We denote by $\|u\|_{L^p}$ the usual Lebesgue norm
in $\R^2$. Moreover, we will write $\|\cdot\|_{H^1(\R)}$,
$\|\cdot\|_{H^1(a,b)}$ to indicate the norms of the Sobolev spaces
of dimension $1$.

However our functional $I_\o$ is defined in the space $\CH$, defined
in \eqref{space}. Its norm will be denoted by $\| \cdot \|_\CH$. In
\cite[Proposition 3.1]{byeon-pre} it is shown that $$\CH \subset \{
u \in C(\R^2): u(0)=0\} \cap L^{\infty}(\R^2).$$

If nothing is specified, strong and weak convergence of sequences
of functions are assumed in the space $\H$.

In our estimates, we will frequently denote by $C>0$, $c>0$ fixed
constants, that may change from line to line, but are always
independent of the variable under consideration. We also use the
notations $O(1), o(1), O(\e), o(\e)$ to describe the asymptotic
behaviors of quantities in a standard way. Finally the letters
$x$, $y$ indicate two-dimensional variables and $r$, $s$ denote
one-dimensional variables.

Let us start with the following proposition, proved in
\cite{byeon, byeon-pre}:

\begin{proposition} $I_\o$ is a $C^1$ functional, and its critical
points correspond to classical solutions of \eqref{equation}.
\end{proposition}

The next result is contained in \cite[Proposition 3.4]{byeon-pre},
and deals with the behavior of $I_{\o}$ under weak limits.

\begin{proposition} \label{prop:weak} Recalling the definition of
$h_u$, \eqref{hu}, let us define: \beq \label{K} K(u)=\frac 1 2
\int_{\RD} h_u^2(x) \frac{u^2(x)}{|x|^2} - 2 N
h_u(x)\frac{u^2(x)}{|x|^2}\, dx. \eeq Then $K$ and $K'$ are weakly
continuous in $\CH$. As a consequence,  $I_\o$ is weak lower
semicontinuous, and $I_{\o}'$ is weakly continuous in $\CH$.
\end{proposition}

Next lemma relates boundedness of sequences in $\H$ and in $\CH$,
and will be very useful in Section 3.

\begin{lemma}\label{lemmino} The map $K$ defined in \eqref{K} is actually
well defined in $\H$ and $K(u_n)$ is bounded if $\|u_n\|$ is
bounded. As a consequence, for any sequence $u_n \in \CH$  such that
$I_{\o}(u_n)$ is bounded from above, $\|u_n\|$ is bounded if and only if $\|u_n\|_{\CH}$ is
bounded. \end{lemma}

\begin{proof}
By \cite{byeon}, we only need to consider the term:

\begin{align*} \int_{\RD} h_u(x)\frac{u^2(x)}{|x|^2}\, dx =  \pi
\int_0^{+\infty} \frac{u^2(r)}{r} \left(\int_0^r s u^2(s)  \, ds\right)  dr \\
\le \pi \int_0^{+\infty} u^2(r) \left(\int_0^r u^2(s) \, ds \right) dr=
\frac{\pi}{2} \left ( \int_0^{+\infty} u^2(r) \, dr \right )^2
\end{align*} Observe now that:

\begin{align*} 2\pi \int_0^{+\infty}  u^2(r) \, dr = \int_{\RD}
\frac{u^2(x)}{|x|}\, dx  \le \int_{B(0,1)} \frac{u^2(x)}{|x|}\,
dx + \int_{\RD\setminus B(0,1)} u^2(x) \, dx \\ \le
C(\|u\|_{L^2}^2+ \|u\|_{L^p}^2), \ p >4, \end{align*} by Holder
inequality. The first assertion of the Lemma follows then from the Sobolev embedding.

Suppose that $u_n$ is bounded in $\H$; then

$$I_{\o}(u_n)= O(1) + \frac{N^2}{2} \int_{\RD}
\frac{u_n^2(x)}{|x|^2}\, dx,$$ and by hypothesis $u_n$ is bounded in
$\CH$. The reverse is trivial.

\end{proof}

The following is a Pohozaev-type identity for problem
\eqref{equation}, see (2.11), (5.6) in \cite{byeon-pre}:

\begin{proposition} \label{poho} For any $u \in \CH$ solution of
\eqref{equation}, the following identity holds:
\[ 
\ird |\n u|^2 dx +\ird \frac{u^2(x)}{|x|^2}\left(h_u(|x|)
-N\right)^2 dx - \frac{p-1}{p+1}\ird |u|^{p+1} dx=0.
\]
\end{proposition}

We now state an inequality which will prove to be fundamental in
our analysis. This inequality is proved in \cite[Proposition
3.5]{byeon-pre}, where also the maximizers are found.

\begin{proposition} For any $u \in \CH$,
\beq \label{ineq} \int_{\R^2} |u(x)|^4 \, dx \le 4 \left
(\int_{\R^2} |\nabla u(x)|^2 \, dx \right )^{1/2} \left
(\int_{\R^2} \frac{u^2(x)}{|x|^2} \left( h_u(|x|) -N\right )^2 dx
\right)^{1/2}. \eeq
\end{proposition}

As commented in the introduction, this paper is concerned with the
boundedness from below of $I_\o$. First of all, let us give a
heuristic derivation of the limit energy functional. Consider
$u(r)$ a fixed function, and define $u_\rho(r)= u(r-\rho)$. Let us
now estimate $I_\o(u_\rho)$ as $\rho \to +\infty$; after the
change of variables $r \to r+\rho$, we obtain:
\begin{align*}
\frac{I_\o (u_\rho)}{2\pi} &= \dis \frac 12 \int_{-\rho}^{+\infty}
(|u'|^2 + \o u^2)(r+\rho) \, dr
\\
& \quad+\dis \frac{1}{8} \int_{-\rho}^{\infty}
\frac{u^2(r)}{r+\rho}\left(\int_{-\rho}^r (s+\rho) u^2(s) \, ds -
2N\right)^2 dr  - \dis \frac{1}{p+1}  \int_{-\rho}^{\infty}
|u|^{p+1}(r+\rho) \, dr.
\end{align*}

We estimate the above expression by simply replacing the
expressions $(r+\rho)$, $(s+\rho)$ with the constant $\rho$;
observe that the estimate is independent of $N$:
$$(2\pi)^{-1} I_\o (u_\rho)  $$$$\sim  \dis \rho \left [ \frac 12 \ir (|u|'^2
+ \o u^2) \, dr
 +\dis \frac{1}{8} \ir
u^2(r) \left(\int_{-\infty}^r  u^2(s) \, ds\right)^2 dr  - \dis
\frac{1}{p+1}  \ir |u|^{p+1} \, dr \right ]$$
$$ =\rho \left [ \frac 12 \ir (|u|'^2
+ \o u^2) \, dr
 +\dis \frac{1}{24} \left ( \ir
u^2 dr \right)^3  - \dis \frac{1}{p+1}  \ir |u|^{p+1} \, dr \right
]. $$

Therefore, it is natural to consider the limit functional $J_\o:
H^1(\R) \to \R$,

\beq \label{def J} J_\o (u)= \frac 1 2  \ir \left ( |u'|^2 + \omega
u^2 \right ) dr + \frac{1}{24} \left ( \ir u^2 \, dr \right)^3 -
\frac{1}{p+1} \ir |u|^{p+1}\, dr. \eeq  Clearly, the
Euler-Lagrange equation of \eqref{def J} is the following limit
problem: \beq \label{limit}
 -u'' + \omega u + \frac 1 4 \left ( \ir
u^2(s)\, ds \right )^2 u = |u|^{p-1} u \quad \hbox{in }\R. \eeq

Let $u$ be a
positive solution of \eqref{limit}, and define $k= \omega + \frac
1 4 \left( \ir u^2\, dr\right )^2$. Then, it is well known that $u(r)= w_k(r-\xi)$ for some $\xi \in \R$, where
\[ 
w_k(r)=
k^{\frac{1}{p-1}}w_1(\sqrt{k}r), \quad \hbox{with} \quad w_1(r)=
\left ( \frac{2}{p+1} \cosh^2 \left (\frac{p-1}{2} r \right)
\right)^{\frac{1}{1-p}}.
\]

We now recall the value of $k$:
$$ k= \omega + \frac 1 4 \left( \ir w_k^2(r)\, dr\right )^2=
\omega + \frac 1 4 k^{\frac{4}{p-1}} \left( \ir w_1^2(\sqrt{k}r)\,
dr\right )^2.$$ A change of variables leads us to the identity:

\beq \label{eq-k} k= \o + \frac 1 4 m^2 k^{\frac{5-p}{p-1}},\ k>0,
\eeq with $m$ is given in \eqref{defm}. Therefore, the existence
of solutions for \eqref{limit} reduces to the existence of
solutions of the algebraic equation \eqref{eq-k}. Moreover, we are
also interested in the energy of those solutions, and whether it
is positive or negative. Those questions have been treated in
\cite[Section 3]{AD}, where the following results were obtained:

\begin{proposition} \label{explicit} Assume $p \in (1,3)$ and take $\o_0$ as in \eqref{w0}. Then:
\begin{enumerate}
\item for any $\o>0$, $J_\o$ is coercive and attain its infimum;

\item There exists $\o_1 > \o_0$ such that for $\o \in (0, \o_1)$,
equation \eqref{eq-k} has two solutions $k_1(\o)<k_2(\o)$ and
$w_{k_1}(r),w_{k_2}(r)$ are the only two positive solutions of
\eqref{limit} (apart from translations);

\item if $\o> \o_0$, $min \, J_{\o} =0$ and the unique minimizer
is $0$.

\item if $\o= \o_0$, $min \, J_{\o}=0$ and is attained at $0$ and
$w_{k_2}$.

\item if $\o \in (0,\o_0)$, $min \, J_{\o}<0$ and the minimizer is
$w_{k_2}$, which is unique (up to change of sign and translation).
\end{enumerate}
\end{proposition}



In this paper we are able to relate $I_\o$ with the limit
functional $J_\o$ in the following way:
$$ \inf I_\o > -\infty \ \Leftrightarrow \ \inf J_\o =0.$$
That is the reason why the explicit value $\o_0$ comes as a
threshold for $I_\o$.

We finish this section with a technical result from
\cite[Proposition 3.7]{AD}, that will be of use later.

\begin{proposition} \label{extra}
Assume $\omega \ge \o_0$, and $u_n \in H^1(\R)$ such that
$J_\o(u_n) \to 0$. There holds
\begin{enumerate}
\item if $\o > \o_0$, then $u_n \to 0$ in $H^1(\R)$; \item if
$\o=\o_0$, then, up to a subsequence, either $u_n \to 0$ or
$u_n(\cdot-x_n) \to \pm w_{k_2}$ in $H^1(\R)$, for some sequence
$x_n \in \R$.

\end{enumerate}
\end{proposition}

\section{Proof of Theorems \ref{teo1}, \ref{teo2}}\label{sezproof}

Our first lemma makes rigorous the heuristic derivation of the
limit functional made in Section 2. Since the functions in $\CH$
must vanish at $0$, we need to truncate our sequence around the
origin. For that purpose, take a Lipschitz continuous function
$\phi_0:\R \to \R$ such that \beq \label{def-phi_0}
\phi_0(r)=\left\{
\begin{array}{ll}
0, & \hbox{if } |r|\le 1,
\\
1, & \hbox{if }  |r| \ge 2,
\end{array}
\right.
\qquad |\phi'_0(r)|\le 1.
\eeq
\begin{lemma}\label{le:asin}
Let $U\in H^1(\R)$ be an even function which decays to zero
exponentially at infinity, and $\phi_0(r)$ as in
(\ref{def-phi_0}). Let us denote $U_\rho(r)= \phi_0(r)U(r-\rho)$. Then
there exists $C>0$ such that:
\[
I_\o(U_\rho)=2 \pi \rho J_\o(U(r))-C+o_\rho(1).
\]
\end{lemma}

\begin{proof}
This estimate has been accomplished in \cite[Lemma 4.1]{AD} for
$N=0$, so we just need to estimate the extra terms:

\begin{align*}
N^2\int_0^{+\infty}  \frac{U_\rho^2(r)}{r} dr-N\int_0^{+\infty}
\frac{U_\rho^2(r)}{r}\left(\int_0^r s U_\rho^2(s) \, ds\right) dr.
\end{align*}
By using the properties of the cut-off function $\phi_0$ we have
\begin{align*}
\int_0^{+\infty}  \frac{U_\rho^2(r)}{r} dr=o_\rho(1)
\end{align*}
and it is not difficult to see that
\begin{align*}
\int_0^{+\infty}  \frac{U_\rho^2(r)}{r}\left(\int_0^r s
U_\rho^2(s) \, ds\right) dr
=\int_{-\infty}^{+\infty}  U^2(r) \left(\int_{-\infty}^r U^2(s) \,
ds\right)dr +o_{\rho}(1)= C + o_{\rho}(1),
\end{align*}
with $C>0$. Hence the conclusion follows.

\end{proof}

In the next proposition we make use of the fundamental inequality
\eqref{ineq} to study the behavior of unbounded sequences with
energy bounded from above.

\begin{proposition} \label{prop-fund}
Assume $\o>0$, and $u_n \in \CH$ such that $\|u_n\| $ is unbounded
but $I_{\o}(u_n)$ is bounded from above. Then, there exists a
subsequence (still denoted by $u_n$) such that:
\begin{enumerate}

\item[i)] for all $\e>0$, $\displaystyle \int_{\e
\|u_n\|^2}^{+\infty} \big(|u_n'|^2 + u_n^2 \big) \, dr \le C$;

\item[ii)] there exists $\d \in (0,1)$ such that  $\displaystyle
\int_{\d \|u_n\|^2}^{\d^{-1} \|u_n\|^2} \big(| u_n'|^2 + u_n^2 \big )\, dr
\ge c>0$;

\item[iii)] $\|u_n\|_{L^2(\RD)} \to +\infty$.

\end{enumerate}

\end{proposition}

\begin{proof}
The proof is quite similar to \cite[Proposition 4.2]{AD}, but
there are some differences at certain points due to
the presence of the singular term. For convenience of the reader,
we reproduce it entirely here. By inequality \eqref{ineq} and
Cauchy-Schwartz inequality, we can estimate:
\begin{align}
I_\o(u) &\ge \frac{\pi}{2}\int_0^{+\infty} \left(|u'|^2 +\o u^2
\right)r \, dr +\frac{\pi}{8}\int_0^{+\infty}
\frac{u^2(r)}{r}\left(\int_0^r s u^2(s) \, ds-2N\right)^2 dr
\nonumber
\\
&\quad+ 2\pi \int_0^{+\infty}\left(\frac{\o}{4}u^2
+\frac{1}{8}u^4-\frac{1}{p+1}  |u|^{p+1}\right) r\, dr.
\label{eq:If}
\end{align}
Define
\[
f:\R_+\to \R,\qquad f(t)= \frac{\o}{4}t^2 +\frac{1}{8}t^4-\frac{1}{p+1} t^{p+1}.
\]
Then, the set $\{t > 0: f(t) < 0\}$ is of the form $(\a,\b)$,
where $\a,\b$ are positive constants depending only on $p, \o$.
Moreover, we denote by $-c_0=\min f<0$.

For each function $u_n$, we define:
 \begin{equation*} 
 A_n=\{x\in \RD : u_n(x)\in (\a,\b)\},\ \rho_n= \sup\{|x|: x\in A_n\}.
 \end{equation*}
\\
With these definitions, we can rewrite \eqref{eq:If} in the form
\begin{equation} \label{eq:s0} I_\o(u_n) \ge
\frac{\pi}{2} \int_0^{+\infty} \left(|u_n'|^2 +\o u_n^2 \right)r
\, dr +\frac{\pi}{8}\int_0^{+\infty}
\frac{u_n^2(r)}{r}\left(\int_0^r s u_n^2(s) \, ds-2N\right)^2 dr
-c_0|A_n|.
\end{equation}
In particular this implies that $|A_n|$ must diverge, and hence
$\rho_n$. This already proves (iii).

\medskip By Strauss Lemma \cite{strauss}, we have
\beq\label{eq:sl} \a\le u_n(\rho_n)\le
\frac{\|u_n\|}{\sqrt{\rho_n}}, \ \Rightarrow \|u_n\|^2\ge
\a^2\rho_n. \eeq

We now estimate the nonlocal term. For that, define
\begin{equation} \label{Bn} B_n=A_n\cap B(0,\g_n), \mbox{ for }
\g_n\in (0,\rho_n) \mbox{ such that }|B_n|=\frac 1 2|A_n|.
\end{equation}
Then $\int_{B_n}  u_n^2(x) \, dx\ge\a^2|B_n|$ diverges, indeed
$$\int_{B_n}  u_n^2(x) \, dx -2N>c|A_n|.$$

We now estimate:
\begin{align}
\int_0^{+\infty}  \frac{u_n^2(r)}{r}\left(\int_0^r s u_n^2(s) \,
ds-2N\right)^2 dr
 & \ge\int_{\g_n}^{+\infty}  \frac{u_n^2(r)}{r}\left(\int_0^r s u_n^2(s) \,
ds-2N\right)^2 dr \nonumber
\\
&\ge \frac{1}{4\pi^2} \int_{\g_n}^{+\infty}
\frac{u_n^2(r)}{r}\left(\int_{B_n}  u_n^2(x) \, dx-2N\right)^2 dr
\nonumber
\\
&\ge c|A_n|^2\int_{\g_n}^{+\infty} \frac{u_n^2(r)}{r}\, dr \nonumber
\\
&\ge c|A_n|^2\int_{A_n\setminus B_n} \frac{u_n^2(x)}{|x|^2}\, dx \nonumber
\\
&\ge c\frac{|A_n|^2}{\rho_n^2}\int_{A_n\setminus B_n} u_n^2(x)\, dx \nonumber
\\
&\ge c\frac{|A_n|^3}{\rho_n^2}.\label{eq:Arho}
\end{align}
Hence, by \eqref{eq:If}, \eqref{eq:sl} and \eqref{eq:Arho}, we get
\[
I_\o(u_n)
\ge c \rho_n +c\frac{|A_n|^3}{\rho_n^2}-c_0|A_n|
=  \rho_n\left(c +c\frac{|A_n|^3}{\rho_n^3}-c_0\frac{|A_n|}{\rho_n}\right).
\]
Observe that  $t\mapsto c+ct^3-c_0 t$ is strictly positive near
zero and goes to $+\infty$, as $t\to +\infty$. Then we can assume,
passing to a subsequence, that $|A_n| \sim \rho_n$. In other
words, there exists $m>0$ such that $ \rho_n|A_n|^{-1} \to m$ as
$n\to +\infty$.
\\
Taking into account \eqref{eq:s0} and \eqref{eq:sl}, we conclude that up to a
subsequence, $\|u_n\|^2 \sim \rho_n$. Moreover, for any fixed
$\e>0$, we have:
\[
C\rho_n \ge \|u_n\|_{L^2}^2  \ge \int_{\e \rho_n}^{+\infty}
u_n^2r \,dr \ge \e \rho_n \int_{\e \rho_n}^{+\infty} u_n^2 \,dr.
\]
An analogous estimate works also for $\int_{\e \rho_n}^{+\infty}
|u_n'|^2 dr$. This proves (i).
\\
\
\\
We now show that for some $\d>0$, $\|u_n\|_{H^1(\d \rho_n,\rho_n)}
\nrightarrow 0$, which implies assertion (ii).
\\
First, recall the definition of $B_n$ and $\g_n$ in \eqref{Bn}.
Then,
$$ \int_{\g_n}^{\rho_n} u_n^2(r) \, dr \ge \rho_n^{-1} \int_{\g_n}^{\rho_n} u_n^2(r) r\,
dr\ge \rho_n^{-1} \int_{A_n\setminus B_n} u_n^2(x) dx \ge
\rho_n^{-1} |A_n\setminus B_n| \a^2>c>0.$$
To conclude it suffices to show that $\g_n \sim \rho_n$. Define
\begin{equation*} 
C_n=B_n\cap B(0,\tau_n), \mbox{ for }
\tau_n\in (0,\gamma_n) \mbox{ such that }|C_n|=\frac 1 2|B_n|.
\end{equation*}

We can repeat the estimate \eqref{eq:Arho} with $A_n$, $B_n$  replaced with
$B_n$, $C_n$ respectively, to obtain that
\[\int_0^{+\infty}  \frac{u_n^2(r)}{r}\left(\int_0^r s u_n^2(s) \,
ds-2N\right)^2 dr \ge c\frac{|B_n|^3}{\gamma_n^2}.\]
Hence,
\[
I_\o(u_n)
\ge c \rho_n +c\frac{|A_n|^3}{\g_n^2}-c_0|A_n|
=  \g_n\left(c \frac{\rho_n}{\g_n} +c\frac{|A_n|^3}{\g_n^3}-c_0\frac{|A_n|}{\g_n}\right).
\]
And we are done since $I_{\o}(u_n)$ is bounded from above.

\end{proof}

\begin{proof}[Proof of Theorem \ref{teo1}]

If $\o\in (0,\o_0)$, then $J_\o(w_{k_2})<0$ (see Proposition
\ref{explicit}): applying Lemma \ref{le:asin} with $U= w_{k_2}$
we conclude assertion (i).

\medskip
We now prove $(ii)$ and $(iii)$.
We denote by $H^1_{0,r}(B(0,R))$ the Sobolev space of radial functions with zero
boundary value and
$$
\CH(B(0,R))=\left\{u\in H^1_{0,r}(B(0,R)):\int_{B(0,R)}\frac{u^2(x)}{|x|^2}dx<+\infty\right\},
$$
endowed by the norm $\|\cdot\|_\CH$.
\\
Fixed $n \in \N$ and given a sequence $v_i\in \CH(B(0,n))$ unbounded
with respect to the norm $\|\cdot\|$, \eqref{eq:s0} implies that $I_{\o}(v_i) \to + \infty$. By Lemma \ref{lemmino}, we conclude that $I_{\o}|_{\CH(B(0,n))}$ is
coercive.
\\
So, there exists $u_n$ a minimizer for $I_{\o}|_{\CH(B(0,n))}$. By
taking absolute value, we can assume that $u_n\ge 0$. Moreover,
$$I_{\o}(u_n) \to \inf I_{\o}, \mbox{ as } n \to +\infty.$$

In the following, $u_n$ may be extended as functions in $\CH$ by setting $u_n(x) = 0$ for $x \in \RD \setminus B(0, n)$.
If $u_n$ is bounded in $H^1(\R^2)$, Lemma \ref{lemmino} implies
that $u_n$ is bounded in $\CH$ and then $I_{\o}(u_n)$ is bounded. In such case we conclude that $\inf I_\o$ is finite. In what follows we assume that $u_n$
is an unbounded sequence in $H^1(\R^2)$, and we shall show that
$I_{\o}(u_n)$ is still bounded for $\o \ge \o_0$.

Our sequence $u_n$ satisfies the hypotheses of Proposition
\ref{prop-fund}, so let $\d>0$ be given by that proposition.

The proof will be divided in several steps.

\

{\bf Step 1:}

$\dis \int_{\frac{\delta}{2}
\|u_n\|^2}^{\frac{2}{\delta} \|u_n\|^2} |u_n|^{p+1}\, dr
\nrightarrow 0$.\\

By Proposition \ref{prop-fund}, i), we have that:

\[ \sum_{k=1}^{[\frac{\delta}{2} \|u_n\|^2]} \int_{\frac{\delta}{2}\|u_n\|^2+k-1}^{\frac{\delta}{2}\|u_n\|^2+k} \left(|u_n'|^2+ u_n^2 \right) dr \le \int_{\frac{\delta}{2}\|u_n\|^2}^{\delta\|u_n\|^2} \left(|u_n'|^2+ u_n^2 \right) dr \le C.\]

Taking the smaller summand in the left hand side we find $x_n$,

$$ \frac{\d}{2} \|u_n\|^2 \le x_n \le \d \|u_n\|^2-1 \ \mbox{ such that }
 \|u_n\|_{H^1(x_n,x_n+1)}^2 \le \frac{C}{\|u_n\|^2}.$$
Reasoning in an analogous way, we can choose $y_n$, $$\d^{-1}
\|u_n\|^2+1 \le y_n \le 2\d^{-1} \|u_n\|^2 \ \mbox{ such that }
\|u_n\|_{H^1(y_n,y_n+1)}^2 \le \frac{C}{\|u_n\|^2}.$$ Observe that
if $\d^{-1} \|u_n\|^2 \ge n$, the choice of $y_n$ can be
arbitrary, but it is unnecessary. Take $\phi_n:[0,+\infty] \to
[0,1]$ be a $C^{\infty}$-function such that
\[
\phi_n(r)=\left\{
\begin{array}{ll}
0, & \hbox{if }r\le x_n,
\\
1, & \hbox{if }  x_n+1\le r \le y_n,
\\
0, & \hbox{if } r\ge y_n+1.
\end{array} \ \ |\phi_n'(r)|\le 2.
\right.
\]
Let $$F(u)=\int_0^{+\infty}\frac{u^2(r)}{r}\left(\int_0^r
su^2(s)ds-2N\right)^2dr.$$ By the choice of $x_n, y_n$ and
Proposition \ref{prop-fund}, i), we have
\begin{align*}
F'(u_n)[\phi_n u_n]
&\ge4\int_0^{+\infty}\frac{u_n^2(r)}{r}\left(\int_0^r su_n^2(s)ds-2N\right)\left(\int_0^r su_n^2(s)\phi_n(s)ds\right)dr\nonumber\\
&\ge-8N\left(\int_{x_n}^{+\infty}u_n^2(r)dr\right)^2>-C.
\end{align*}
It follows that
\begin{align*}
0&=I_\o'(u_n)[\phi_n u_n] \ge 2\pi \int_{x_n}^{y_n} \left(|u'_n|^2 +\o
u_n^2\right)r \, dr
 - 2\pi \int_{x_n}^{y_n} |u_n|^{p+1}r \, dr +O(1)
\\
& \ge \|u_n\|^2 \left ( \frac{\d}{2} \int_{x_n}^{y_n} \left(|u'_n|^2 +\o u_n^2\right)  dr
 - \frac{2}{\d} \int_{x_n}^{y_n} |u_n|^{p+1} \, dr \right)
+O(1).
\end{align*}
This, together with the fact that $\|u_n\|_{H^1(x_n, y_n)}$ does
not tend to zero, allows us to conclude the proof of Step 1.

\

{\bf Step 2:} Exponential decay.

\

At this point we can apply the concentration-compactness principle
(see \cite[Lemma 1.1]{lions}); there exists $\s>0$ such that
\[
\sup_{\xi\in [x_n,\ y_n]}\int_{\xi-1}^{\xi+1}u_n^2\, dr\ge 2\s>0.
\]
Let us define:

\beq \label{def-D_n} D_n=\left\{\xi >0:
\int_{\xi-1}^{\xi+1} \left(|u_n'|^2+ u_n^2\right) dr\ge \s\right\} \neq
\emptyset, \ \hbox{and} \ \xi_n=\max D_n \in [x_n, n+1). \eeq
Let us observe that $\xi_n\sim \|u_n\|^2$; indeed $\xi_n \ge x_n \ge c\|u_n\|^2$ and, moreover,
\[
\|u_n\|^2\ge c\int_{\xi_n-1}^{\xi_n+1}\left(|u_n'|^2+ u_n^2\right) r\  dr \ge c (\xi_n-1) \int_{\xi_n-1}^{\xi_n+1}\left(|u_n'|^2+ u_n^2\right) \,  dr
\ge  c (\xi_n-1).
\]

By definition, $\int_{\zeta-1}^{\zeta+1}(|u_n'|^2+u_n^2)\, dr <
\s$ for all $\zeta > \xi_n$. By embedding of
$H^1(\zeta-1,\zeta+1)$ in $L^{\infty}$, $0 \le u_n(\zeta) < C \sqrt{\s}
$ for any $\zeta > \xi_n$. From this we will get exponential decay of $u_n$. Indeed, $u_n$ is a solution of
\[ 
 - u_n''(r) - \frac{u_n'(r)}{r} + \o u_n(r) + f_n(r) u_n(r)  = u_n^p(r),
 \]
with
$$  f_n(r)=  \frac{(h_n(r)-N)^2}{r^2} + \int_r^n \frac{h_n(s)-N}{s} u_n^2(s) \, ds, \ \ \  h_n(r)= \frac 12\int_0^r s u_n^2(s)  \, ds.$$
If $r>\d
\|u_n\|^2$, again by Proposition \ref{prop-fund}, i), we see that $\int_r^n \frac{N}{s} u_n^2(s) \, ds=o(1)$. Then, by taking smaller $\s$, if necessary, we can
conclude that there exists $C>0$ such that
$$ |u_n(r)| <C {\rm exp}\left(- \sqrt{\frac{\o}{2}} (r-\xi_n)\right), \quad
\mbox{ for all } r>\xi_n. $$

The local $C^1$ regularity theory for the Laplace operator (see
\cite[Section 3.4]{gilbarg}) implies a similar estimate for
$u_n'(r)$. In other words,
\[ 
|u_n(r)| + |u_n'(r)|<C {\rm exp}\left(- \sqrt{\frac{\o}{2}}
(r-\xi_n)\right), \quad \mbox{ for all } r>\xi_n.
\]

\

{\bf Step 3:} Splitting of $I_\o(u_n)$.

Reasoning as in the beginning of Step 1, we can take $z_n$:
 $$ \xi_n - 3 \|u_n\| \le z_n \le \xi_n-2\|u_n\| \ \mbox{ with } \|u_n\|^2_{H^1(z_n, z_n+1)} \le \frac{C}{\|u_n\|}. $$

Define $\psi_n:[0,+\infty] \to [0,1]$ be a smooth function such that
\[ 
\psi_n(r)=\left\{
\begin{array}{ll}
0, & \hbox{if }r\le z_n,
\\
1, & \hbox{if }  r \ge z_n+1,
\end{array}  \ \ |\psi_n'(r)|\le 2.
\right.
\]

We claim that \beq\label{eq:ICR} I_\o(u_n)\ge I_\o(u_n
\psi_n)+I_\o\left(u_n(1-\psi_n)\right) +c \|u_n
(1-\psi_n)\|_{L^2(\RD)}^2+ O(\|u_n\|). \eeq

This estimate has been accomplished in \cite{AD} for $N=0$.
Therefore we just need to estimate the two new terms; it is easy
to get that
%
%

\begin{align*}
\int_0^n \frac{u_n^2(r)}{r} dr=\int_0^n
\frac{u_n^2(r)\psi^2_n(r)}{r} dr+\int_0^n
\frac{u_n^2(r)(1-\psi_n(r))^2}{r} dr+o(1).
\end{align*}

Moreover,
\begin{align*}
\int_0^n &\frac{u_n^2(r)}{r} \left(\int_0^r s u_n^2(s)\, ds\right)dr=\int_0^n \frac{u_n^2(r)\psi^2_n(r)}{r} \left(\int_0^r s\psi_n^2(s) u_n^2(s)\, ds\right)dr\\
&\qquad+\int_0^n \frac{u_n^2(r)(1-\psi_n(r))^2}{r} \left(\int_0^r s(1-\psi_n(s))^2 u_n^2(s)\, ds\right) dr\\
&\qquad+\underset{(I)}{\underbrace{\int_0^n \frac{u_n^2(r)\psi_n^2(r)}{r} \left(\int_0^r su_n^2(s)(1-\psi^2_n(s))\,ds\right)dr}}\\
&\qquad {+ \underset{(II)}{\underbrace{2\int_0^n \frac{u_n^2(r)\psi_n^2(r)}{r} \left(\int_0^rs u_n^2(s)\psi_n(s)(1-\psi_n(s))\,ds\right)dr}}}\\
&\qquad{+ \underset{(III)}{\underbrace{\int_0^n \frac{u_n^2(r)(1-\psi_n(r))^2}{r} \left(\int_0^rs u_n^2(s)\psi_n^2(s)\,ds\right)dr}}}\\
&\qquad+\underset{(IV)}{\underbrace{2\int_0^n \frac{u_n^2(r)(1-\psi_n(r))^2}{r} \left(\int_0^rs u_n^2(s)\psi_n(s)(1-\psi_n(s))\,ds\right)dr}}\\
&\qquad +\underset{(V)}{\underbrace{2\int_0^n \frac{u_n^2(r)(1-\psi_n(r))\psi_n(r)}{r}\left(\int_0^r s u_n^2(s)\,ds\right) dr}}.\\
\end{align*}

We now observe that (I) $\dots$ (V) are bounded, as follows:
\begin{align*}
(I)& \le  \int_{z_n}^{n} \frac{u_n^2(r)}{r}\left(\int_0^{z_n+1} s
u_n^2(s)\,ds\right) dr \le \frac{C\|u_n\|^2}{z_n}=O(1),
\\ (II) &  \le 2\int_{z_n}^{n}
 \frac{u_n^2(r)}{r}\left(\int_{z_n}^{z_n+1} s u_n^2(s)\,ds\right) dr=O(1),
\end{align*}
and the other terms can be estimated similarly. Therefore, we
conclude the proof of \eqref{eq:ICR}.

\

{\bf Step 4:} The following estimate holds:

\beq \label{comparison} I_\o(u_n \psi_n) = 2\pi \xi_n J_\o(u_n
\psi_n) +O(\|u_n\|). \eeq

\

In \cite{AD} this estimate was made for $N=0$. So we just need to check the new nonlocal terms

\begin{align*}
\int_0^n  \frac{(u_n \psi_n)^2(r)}{r} dr&\le \frac{C}{z_n}=o(1),
\\
\int_0^n  \frac{(u_n \psi_n)^2(r)}{r}\left(\int_0^r s (u_n
\psi_n)^2(s) \, ds\right) dr&\le  \int_{z_n}^n
\frac{u_n^2(r)}{r}\left(\int_{z_n}^r s u_n^2(s)\, ds\right)dr \\ &\le
\left(\int_{z_n}^n  u_n^2(r)\, dr\right)^2 =O(1).
\end{align*}

{\bf Step 5:} Conclusion for $\o > \o_0$.

\

By \eqref{eq:ICR} and \eqref{comparison}, we have
\beq\label{eq:ICR2} I_\o(u_n)\ge 2 \pi \xi_n J_\o(u_n
\psi_n)+I_\o(u_n(1-\psi_n)) +c \|u_n (1-\psi_n)\|_{L^2(\RD)}^2+
O(\|u_n\|). \eeq

Recall that $\|u_n \psi_n\|^2_{H^1(\R)} \ge \s>0$. By Proposition
\ref{extra}, we have that $J_\o(u_n \psi_n) \to c>0$, up to a
subsequence. Since $\xi_n \sim \|u_n\|^2$, it turns out from
\eqref{eq:ICR2} that $I_\o(u_n) > I_\o(u_n(1-\psi_n))$, which is a
contradiction with the definition of $u_n$. Therefore, $u_n$ needs
to be a bounded sequence and, in particular, $\inf I_\o>-\infty$.

 Let us now show that $I_{\o}$ is coercive. Indeed,
take $u_n \in \CH$ an unbounded sequence, and assume that
$I_{\o}(u_n)$ is bounded from above. By Lemma \ref{lemmino},
$\|u_n\|$ is unbounded, so that Proposition \ref{prop-fund},
(iii), shows us that $I_{\hat{\o}}(u_n) \to -\infty$ for any
$\o_0<\hat{\o} < \o$, a contradiction.

\

{\bf Step 6:} Conclusion for $\o = \o_0$.

\

As above, \eqref{eq:ICR2} gives a contradiction unless $J_\o(u_n
\psi_n) \rightarrow 0$.
Proposition \ref{extra} now implies that
$\psi_n u_n (\cdot - t_n)\to w_{k_2}$ up to a subsequence, for some
$t_n \in (0,+\infty)$. Since $\xi_n \in D_n$ (recall its definition
in \eqref{def-D_n}), we have that $|t_n-\xi_n|$ is bounded. With
this extra information, we have a better estimate of the decay of
the solutions: indeed,

\beq \label{decay2} |u_n(r)| + |u_n'(r)|<C {\rm exp}\left(-
\sqrt{\frac{{\o}}{2}} |r-\xi_n|\right),\quad \mbox{ for all }
r>\xi_n-2\|u_n\|. \eeq

This allows us to do the cut-off procedure in a much more accurate
way. Indeed, take $\tilde{z}_n= \xi_n-\|u_n\|$.  Then,
\eqref{decay2} implies that

\beq \label{new} \|u_n\|^2_{H^1(\tilde{z}_n, \tilde{z}_n+1)} \le C
{\rm exp}(- \sqrt{\frac{{\o}}{2}} \|u_n\|). \eeq

Define $\tilde{\psi}_n:[0,+\infty] \to [0,1]$ accordingly:
\[
\tilde{\psi}_n(r)=\left\{
\begin{array}{ll}
0, & \hbox{if }r\le \tilde{z}_n,
\\
1, & \hbox{if }  r \ge \tilde{z}_n+1,
\end{array} \ \ |\tilde{\psi}_n'(r)|\le 2.
\right.
\]

The advantage is that, in the estimate of $I_\o(u_n)$, now the
errors are exponentially small. Indeed, by repeating the estimates
of Step 3 with the new information \eqref{new}, we obtain:
\[
 I_\o(u_n)\ge I_\o(u_n
\tilde\psi_n)+I_\o(u_n(1-\tilde\psi_n)) +c \|u_n (1-\tilde\psi_n)\|_{L^2(\RD)}^2+
O(1),
\]

Then,
\begin{align*}
I_\o(u_n) & \ge I_\o(u_n \tilde{\psi}_n)+I_\o\big(u_n(1-\tilde{\psi}_n)\big)
+c \|u_n (1-\tilde{\psi}_n)\|_{L^2(\RD)}^2+ O(1)
\\
& = 2 \pi\xi_n J_\o(u_n
\tilde{\psi}_n)+I_\o\big(u_n(1-\tilde{\psi}_n)\big) +c \|u_n
(1-\tilde{\psi}_n)\|_{L^2(\RD)}^2+ O(1)
\\
&\ge  I_{(\o+2c)} \big(u_n(1-\tilde{\psi}_n)\big)+ O(1).
\end{align*}
But, by Step 5, we already know that $I_{(\o+2c)}$ is bounded from below, and hence $\inf I_{\o_0} > -\infty$.

Finally, by applying Lemma \ref{le:asin} to $U=w_{k_2}$ we readily get
that $I_{\o_0}$ is not coercive.

\end{proof}

\begin{proof}[Proof of Theorem \ref{teo2}] We shall prove each assessment separately.

{\bf Proof of (ii).} First, we observe that since $\inf
I_{\o_0}<0$, there exists $\tilde{\o}>\o_0$ such that $\inf I_\o<0$
if and only if $\o\in (\o_0,\tilde{\o})$. Since, by Theorem
\ref{teo1} and Proposition \ref{prop:weak}, $I_\o$ is coercive and
weakly lower semicontinuous, we infer that the infimum is
attained at a negative value. This gives the first solution $u_1$.
\\
Clearly, $0$ is a local minimum for $I_\o$, and $I_\o(u_1)<0$. Then, the functional satisfies the geometrical
assumptions of the Mountain Pass Theorem, see \cite{ar}. Since
$I_\o$ is coercive, (PS) sequences are bounded. By the compact
embedding of $\Hr$ into $L^{p+1}(\R^2)$ and Proposition
\ref{prop:weak}, standard arguments show that $I_\o$ satisfies the
Palais-Smale condition and so we find a second solution which is
at a positive energy level.

\
\\
{\bf Proof of (iii).} Let now consider $\o\in (0,\o_0)$.
Performing the rescaling $u\mapsto u_\o=\sqrt{\o}\ u(\sqrt{\o}\ \cdot)$,
we get
\begin{align*}
I_\o(u_\o) & =  \o \left[ \frac 12 \int_{\R^2} \left(|\nabla u|^2
+  u^2\right) dx + \frac{1}{8} \int_{\R^2}
\frac{u^2(x)}{|x|^2}\left(\int_0^{|x|} s u^2(s) \, ds-2N\right)^2
dx \right.
\\
 & \left.\qquad  -
\frac{\o^\frac{p-3}{2}}{p+1}  \int_{\R^2} |u|^{p+1} \, dx\right].
\end{align*}
Define $\l=\o^\frac{p-3}{2}$ and $\mathcal{I}_\l:H(\R^2)\to \R$ as
\begin{align*}  \mathcal{I}_\l(u)= \Phi(u) -
\frac{\l}{p+1}  \int_{\R^2} |u|^{p+1} \, dx, \end{align*}
 with
\begin{align*}
\Phi(u)  =  \frac 12 \| u\|^2 + \frac{1}{2} \int_{\R^2}
\frac{u^2(x)}{|x|^2} (h_u(|x|)-N)^2
dx \\
= \frac 12 \| u\|^2 + K(u) +\frac{N^2}{2} \int_{\R^2}
\frac{u^2(x)}{|x|^2}\,
dx, \end{align*}
where $K$ is as defined in \eqref{K}. Then $\I_\l$ satisfies the geometrical assumptions of the Mountain
Pass Theorem. The main problem here is that we do not know whether a (PS) sequence could be unbounded.

By Lemma \ref{lemmino}, the functional $\Phi : \CH \to \R$ is coercive. Then we can use \cite[Theorem 1.1]{jeanjean} to obtain a bounded
Palais-Smale sequence $u_n \in \CH$ for almost every $\l$. Passing to a subsequence, we can assume that $u_n \weakto u$; Proposition
\ref{prop:weak} and standard arguments imply that $u$ is a
critical point of ${\mathcal I}_\l$. Making the change of
variables back we obtain a solution of \eqref{equation} for almost
every $\o\in (0,\o_0)$.

\
\\
Finally, in order to find positive solutions of \eqref{equation},
we simply observe that the above arguments apply to the functional
$I_\o^+:\CH \to \R$
\begin{align*}
I_\o^+(u) & =  \frac 12 \int_{\R^2} \left(|\nabla u|^2 + \o
u^2\right) dx + \frac{1}{8} \int_{\R^2}
\frac{u^2(x)}{|x|^2}\left(\int_0^{|x|} s u^2(s) \, ds-2N\right)^2
dx
\\
 & \quad  -
\frac{1}{p+1}  \int_{\R^2} (u^+)^{p+1} \, dx.
\end{align*}
Due to the maximum principle, the critical points of $I_\o^+$ are
positive solutions of \eqref{equation}.

{\bf Proof of (i).} This part happens to be quite delicate, compared to the case $N=0$ studied in \cite{AD}. Let $u$ be a
solution of \eqref{equation}. If we multiply \eqref{equation} by $u$
and integrate, we get

\begin{align}\label{eq:j1}
0&= \ird \left(|\n u|^2+\o u^2\right) dx
+3\ird \frac{u^2(x)}{|x|^2}\left(h_u(|x|) -N\right)^2 dx\nonumber\\
&\quad+2N\ird \frac{u^2(x)}{|x|^2}\left(h_u(|x|)-N \right) dx
- \ird |u|^{p+1} dx.
\end{align}

From \eqref{eq:j1} and the Pohozaev identity (Proposition
\ref{poho}), we obtain that, for any $l>0$,
\begin{align}\label{eq:non3}
0&= (l+1)\ird |\n u|^2dx+\o \ird u^2dx
+(l+3)\ird \frac{u^2(x)}{|x|^2}\left(h_u(|x|) -N\right)^2 dx\nonumber\\
&\quad+2N\ird \frac{u^2(x)}{|x|^2}\left(h_u(|x|)-N \right)  dx
-\left(\frac{p-1}{p+1}l+1\right)\ird |u|^{p+1} dx.
\end{align}
 By using \eqref{ineq} in \eqref{eq:non3},
\begin{align}\label{eq:non6}
0&\ge  \ird \left(|\n u|^2+\o u^2\right) dx
+ 3\ird \frac{u^2(x)}{|x|^2}\left(h_u(|x|) -N\right)^2 dx\nonumber\\
&\quad +2N\ird \frac{u^2(x)}{|x|^2}\left(h_u(|x|)-N \right) dx-
\left(\frac{p-1}{p+1}l+1\right)\ird |u|^{p+1} dx+\frac{l}{2}\ird
|u|^4dx.
\end{align}
We can estimate
\begin{align}\label{eq:non7}
 &3\ird \frac{u^2(x)}{|x|^2}\left(h_u(|x|) -N\right)^2 +2N\ird \frac{u^2(x)}{|x|^2}\left(h_u(|x|)-N \right) dx
\nonumber\\
&=\ird \frac{u^2(x)}{|x|^2}(h_u(|x|) -N)(3h_u(|x|)-N ) dx\nonumber\\
&\ge -\frac{N^2}{3}\int_{\{N/3\le h_u\le N\}}\frac{u^2(x)}{|x|^2}dx,
\end{align}
where $\{N/3\le h_u\le N \}=\{r_0\le|x|\le r_1\}$ with
$h_u(r_0)=N/3$ and $h_u(r_1)=N$ (here we have used that $h_u$ is
increasing in $r$). For any $r>0$, by the definition of $h_u$  we
have
\[ 
4\pi h_u(r)=\int_{B_r}u^2(x)dx\le C r\left
(\int_{B_r}u^4(x)dx\right )^{1/2}.
\]
Then
\begin{align}\label{eq:non5}
\int_{B_r} h_u^2(|x|)\frac{u^2(x)}{|x|^2}dx\le C\int_{B_r}u^2(x)\left(\int_{B_{|x|}}u^4(y)dy\right)dx
\\ \le C\int_{B_r}u^2(x)dx\int_{B_r}u^4(x)dx\nonumber \le C h_u(r)\int_{B_r}u^4(x)dx.
\end{align}
We now apply \eqref{eq:non5} to estimate
\begin{align}\label{eq:non8}
\int_{\{N/3\le h_u\le N\}}\frac{u^2(x)}{|x|^2}dx&\le C \int_{\{N/3\le
h_u\le N\}}h_u^2(|x|)\frac{u^2(x)}{|x|^2}dx
 \le C \int_{B_{r_1}}h^2_u(|x|)\frac{u^2(x)}{|x|^2}\, dx\nonumber\\
 &\le Ch_u(r_1) \int_{B_{r_1}}u^4(x)dx\le C\ird u^4dx.
\end{align}
We apply \eqref{eq:non7} and \eqref{eq:non8} in \eqref{eq:non6}:
\[
0\ge \ird |\n u|^2dx+\o \ird u^2dx-c\ird u^4dx+\frac l2\ird u^4
-\left(\frac{p-1}{p+1}l+1\right)\ird |u|^{p+1} dx.
\]
Therefore it suffices to take $l$ so that $-c+\frac l 2=1$, and
then to take $\o$ so that the function
\begin{align*}
 s\rightarrow  \o s^2+s^4-\left(\frac{p-1}{p+1}l+1\right)|s|^{p+1}
\end{align*}
is non-negative for any $s$. Therefore $u$ must be identically
equal to zero.

\end{proof}

\end{document}